\documentclass[a4paper,12pt]{article}
\usepackage[english]{babel}
\usepackage{amsfonts}
\usepackage{tikz, tikz-cd}
\usepackage{amssymb,amsmath,amsthm,eucal,amscd}
\usepackage[left=1cm, right=1cm, top=1cm, bottom=2cm]{geometry}

\usepackage{amsrefs}

\usepackage{mathtools}
\usepackage{url}

\usepackage{hyperref}
\usepackage{xcolor}
\hypersetup{
    colorlinks,
    linkcolor={blue!90!black},
    citecolor={blue!90!black},
    urlcolor={blue!90!black}
}

\mathtoolsset{showonlyrefs}
\newcommand{\Z}{\mathbb Z}
\newcommand{\N}{\mathbb N}
\newcommand{\Q}{\mathbb Q}
\newcommand{\R}{\mathbb R}
\newcommand{\C}{\mathbb C}
\renewcommand{\P}{\mathbb P}
\newcommand{\mc}{\mathcal}
\newcommand{\sgn}{{\rm sgn}}
\renewcommand{\phi}{\varphi}

\renewcommand{\geq}{\geqslant}
\renewcommand{\ge}{\geqslant}
\renewcommand{\le}{\leqslant}

\theoremstyle{plain}
\newtheorem{thm}{Theorem}[section]
\newtheorem{lm}[thm]{Lemma}
\newtheorem{cor}[thm]{Corollary}
\newtheorem{pr}[thm]{Proposition}

\theoremstyle{remark}
\newtheorem{rem}[thm]{Remark}
\newtheorem{ex}[thm]{Example}
\theoremstyle{definition}
\newtheorem{definition}[thm]{Definition}

\begin{document}
\date{}

\author{
	Yury Ustinovskiy\footnote{Princeton University, E-mail: \texttt{yuryu@math.princeton.edu}}
	\and
	Grigory Solomadin\footnote{Moscow State University, E-mail: \texttt{grigory.solomadin@gmail.com}}
}
\title{Projective toric generators in the unitary cobordism ring.}

\maketitle
\abstract{
By the classical result of Milnor and Novikov, the unitary cobordism ring is isomorphic to a graded polynomial ring with countably many generators: $\Omega^U_*\simeq \Z[a_1,a_2,\dots]$, ${\rm deg}(a_i)=2i$. In this paper we solve the well-known problem of constructing geometric representatives for $a_i$ among smooth projective toric varieties, $a_n=[X^{n}], \dim_\C X^{n}=n$. Our proof uses a family of equivariant modifications (birational isomorphisms) $B_k(X)\to X$ of an arbitrary complex manifold $X$ of (complex) dimension $n$ ($n\geq 2$, $k=0,\dots,n-2$). The key fact is that the change of the Milnor number under these modifications depends only on the dimension $n$ and the number $k$ and does not depend on the manifold $X$ itself.
}

\section{Introduction}\label{intro}
Unitary cobordism theory is an extraordinary cohomology theory, which has numerous applications in algebraic topology due to its intrinsic geometric nature. We refer to~\cite{bu-12} and references therein for a thorough survey of the subject. The coefficient ring for this cohomology theory is $U_*(pt)=\Omega^U_*$ and construction of distinguished representatives in every class $a\in\Omega^U_*$ is an important problem in the theory of unitary cobordisms. In the late 1950's F.~Hirzebruch raised a question (\cite{hi-60}), which can be reformulated in terms of unitary cobordisms as follows: when can a class $a\in\Omega^{U}_{*}$  be represented by a \emph{connected} algebraic variety? Already in real dimension $4$ this problem turns out to be connected with deep questions of algebraic geometry and is still open in general.

A challenging task closely related to Hirzebruch's original question is to describe particular manifolds representing the generators of the ring $\Omega^U_*$. There are both classical theorems (see, e.g., Stong's monograph \cite{st-68}), as well as very recent results (cf.\,\cite{wi-13}), concerning this problem. First of all, recall that over $\Q$ the unitary cobordism ring is a polynomial ring generated by the projective spaces $\{\C P^k\}_{k=1}^{\infty}$ (cf.\,\cite[Chapter VII]{st-68}):
\[
\Omega^U_*\otimes \Q\simeq \Q[[\C P^1], [\C P^2], \dots].
\]
So, over the rational numbers the ring $\Omega^U_*\otimes \Q$ has an explicit and simple set of generators. However, the search for multiplicative generators for $\Omega^U_*$ itself turns out to be difficult due to divisibility relations among Chern numbers of stably complex manifolds. As a result many approaches to this and related problems lead to subtle number-theoretical questions (\cite{wi-13, lu-pa-14}). A general result proved independently by Milnor and Novikov gives an efficient (necessary and sufficient) criterion that ensures that a sequence of stably complex manifolds generates the cobordism ring $\Omega^U_*$:

\begin{thm}[Milnor, Novikov, \cite{no-62, st-68}]\label{mn}
The cobordism class of a stably complex manifold $X^{2n}$ may be taken to be a $2n$-dimensional generator iff:
\[
s_n(X^{2n})=
\begin{cases}
\pm 1,\ n\neq p^k-1& \mbox{for any prime number }p; \\
\pm p,\ n=p^k-1& \mbox{for some prime number }p.\end{cases}
\]
\end{thm}
Here $s_n(X^{2n})$ is the Milnor number of $X^{2n}$ (also sometimes referred to as the \emph{top characteristic number}):
\[
s_n(X^{2n}):=\langle t_1^n+\dots+t_n^n, [X^{2n}]\rangle,
\]
where $t_1,\dots,t_n$ are Chern roots of the stably complex tangent bundle $TX$.

An explicit set of generators of the ring $\Omega^U_*$ was first given by Milnor in the 1960's. He proved that suitable linear combinations of bidegree-(1,1) hypersurfaces $H_{i,j}\subset \C P^i\times \C P^j$ generate the unitary cobordism ring. With Milnor's construction one can produce (not connected) algebraic generators of the ring $\Omega^U_*$, \cite{mi-65},~\cite{st-68}. In the category of stably complex manifolds one can define an operation of connected sum. Hence, there exist \emph{connected} stably complex generators of $\Omega^U_*$. In spite of these observations it is natural to ask: \emph{Is it possible to choose connected algebraic varieties as polynomial generators of the ring $\Omega^U_*$?} A positive answer to this question was given by Johnston~\cite{jo-04} in 2004. He produced generators of  $\Omega^U_*$ via sequences of blow-ups of projective varieties along complete intersections in the exceptional divisors.

Another problem related to the construction of ``nice'' generators of $\Omega^U_*$ is the search for polynomial generators with a large symmetry group. Buchstaber, Ray (\cite{bu-ra-98}) and Buchstaber, Panov, Ray (\cite{bu-pa-ra-07}) used a class of \emph{quasitoric} manifolds to construct a representative in any class of the ring $\Omega^U_*$. Quasitoric manifolds (toric varieties in sense of Davis-Januszkiewicz~\cite{da-ja-91}) are connected stably complex manifolds $M^{2n}$ with an effective action of a half-dimensional torus $(S^1)^n$. The key tool used in these papers is the operation of equivariant \emph{box sum} (cf.\,\cite[\S 5]{bu-pa-ra-07}) in the category of quasitoric manifolds, which generalises the operation of connected sum.

From now on $n$ stands for the \emph{complex} dimension of the corresponding complex manifold.

\emph{Toric varieties} are algebraic varieties $X$ with an effective action of an algebraic torus $(\C^*)^{{\rm dim}_\C X}$ having an open dense orbit. Any smooth, compact toric variety represents some element in $\Omega^U_*$, and one can ask whether the generators of the unitary cobordism ring can be represented by toric manifolds. This problem turns out to be more subtle, since the operation of connected sum is not well-defined in the category of toric varieties. A significant advance towards the positive answer to this question was recently made by Wilfong in \cite{wi-13}. He constructed toric generators for $\Omega^U_*$ (i.e., manifolds satisfying the conditions of Theorem~\ref{mn}) in all odd dimensions and in dimensions one less then a power of a prime. In his construction Wilfong started with a certain tower of projectivised bundles over $\C P^1$ (generalised Bott towers) and used sequences of equivariant blow-ups in $(\C^*)^n$-fixed points and along invariant rational curves.

%

To construct polynomial generators for $\Omega^U_*$, it remains to find toric manifolds $X^{n}$ satisfying the conditions of Theorem~\ref{mn} in all even dimensions $n$ such that $n+1$ is not a power of a prime (the minimal such $n$ is 14). 
Our approach to this problem combines the ideas of~\cite{jo-04} and~\cite{wi-13}. Similarly to Johnston and Wilfong we study blow-ups of algebraic complex manifolds. However, unlike~\cite{wi-13} we define a family of birational modifications $B_k(X)\to X$ over any complex manifold $X$. Equivariant versions of these morphisms are well-defined in the category of toric varieties (see Section \ref{sec:equivariant_modification}). The modification $B_k$ is a \emph{sequential blow-up} $B_k(X)\to Bl_x{X}\to X$, where $Bl_x{X}\to X$ is a blow-up at a point and $B_k(X)\to Bl_x{X}$ is a blow-up along a submanifold $\C P^{k}\subset E_x$ of an exceptional divisor $E_x\subset Bl_x{X}$. There is freedom in the choice of the point $x\in X$ and the subvariety $\C P^{k}\subset E_x$ to blow up, however, as it turns out, the difference $[B_k(X)]-[X]$ in $\Omega^U_*$ depends only on the numbers $n$ and $k$ and does not depend on $X$. In particular, the numbers $s_{k,n}=s_n(B_k(X))-s_n(X)$ are well-defined. The latter measure the change of the Milnor number of an $n$-dimensional manifold $X$ under the operations $B_k$. The crucial property of the numbers $s_{k,n}$ in the dimensions of interest to us is that $\mathrm{gcd}(s_{0,n},\dots,s_{n-2,n})=1$. This fact allows us to start from a smooth projective toric variety $X$ with $s_n(X)$ large enough and apply a sequence of modifications $B_{k_i}$ to get a smooth projective toric variety with the Milnor number~1. The main result of this paper can be formulated as follows:

\begin{thm}
There exist smooth projective toric varieties $\{X^{n}\}^\infty_{n=1}$ of complex dimension $n$ such that 
\[
\Omega^{U}_{*}=\Z[[X_1],[X_2],\dots].
\]
\end{thm}
In the category of toric varieties the modifications $B_k$ have a further interesting property: toric manifolds $B_k(X)$ and $B_{n-k-2}(X)$ have combinatorially equivalent moment polytopes, provided that the blown-up invariant subvarieties $\C P^{k}\subset E$ and $\C P^{n-k-2}\subset E$ of an exceptional divisor $E\subset Bl_x{X}$ are properly chosen (see Proposition \ref{equiv}). This property allows us to formulate a characterisation theorem for the 2-parameter Todd genus $\chi_{a,b}$ in terms of \emph{combinatorial rigidity} (see Theorem \ref{thm:2todd_toric}).

%
%
%
%
%
%

\section{Modifications of complex manifolds and relations in the unitary cobordism ring}
\subsection{Projectivisations of complex vector bundles}\label{sec:projectivisations}
We start this Section with a brief review of important facts about projectivisations of vector bundles and their stably complex structures.

\begin{thm}[Leray-Hirsch, see {\cite[$\S$ 20]{bo-tu-82}}]
Let $\xi\to B$ be a complex $(n-k+1)$-dimensional vector bundle over a manifold $B$ of real dimension $2k$. Consider the fiberwise projectivisation $p\colon \mathbb P(\xi)\to B$ of the bundle $\xi$. Denote by $v=c_1(\gamma)\in H^2(\mathbb P(\xi),\Z)$ the first Chern class of the fiberwise line bundle $\gamma=\mc O(1)$ over $\P(\xi)$. Then there is an isomorphism of graded rings:
\begin{equation}\label{eq:leray-hirsch}
H^{*}(\mathbb{P}(\xi),\Z)\cong H^{*}(B,\Z)[v]/(v^{n-k+1}+c_{1}(\xi)v^{n-k}+\dots+c_{n-k+1}(\xi)).
\end{equation}
\end{thm}

This particular statement is a corollary of the general Leray-Hirsch theorem about the degeneration of the Serre spectral sequence and the vanishing of $c_{n-k+1}(p^*\xi\otimes\mc O(1))=0$. The latter holds, since the bundle $p^*\xi\otimes\mc O(1)$ splits off a trivial line bundle.

Note that $H^*(\mathbb{P}(\xi),\Z)$ is a $H^*(B,\Z)$-module via homomorphism $p^*\colon H^*(B,\Z)\to H^*(\mathbb{P}(\xi),\Z)$ and isomorphism~\eqref{eq:leray-hirsch} respects this structure. Further we identify $H^*(B,\Z)$ with its image $p^*(H^*(B,\Z))\subset H^*(\mathbb{P}(\xi),\Z)$. Assume that $B$ is compact and oriented. Then for any class $\omega\in H^{2(n-l)}(B,\Z)$ one can integrate the top class $\omega\cdot v^l$ over $\mathbb{P}(\xi)$. With the use of the Leray-Hirsch theorem one can show that (see~\cite[\S 2.2]{sc-14})

\begin{equation}\label{eq:integration}
\langle{\omega\cdot v^l, [\mathbb P(\xi)}]\rangle=\langle\omega\cdot c^{-1}(\xi), [B]\rangle,
\end{equation}
where $c^{-1}(\xi)\in H^*(B,\Z)$ is the total Segre class, i.e., the multiplicative inverse of the total Chern class $c(\xi)=1+c_1(\xi)+\dots+c_{n-k+1}(\xi)$.

If the base manifold $B$ has a complex structure, the projectivisation $\P(\xi)$ is equipped with a canonical stably complex structure
\begin{equation}\label{eq:stably-complex-standard}
T\mathbb P(\xi)\oplus\C\simeq (p^*\xi\otimes \gamma)\oplus p^*TB.
\end{equation}
Using the short exact sequence $T_{\rm vert}\mathbb P(\xi)\to T\mathbb P(\xi)\to p^*TB$ and the isomorphism $T_{\rm vert}\mathbb P(\xi)\simeq\mathop{\rm Hom}(\mc O(-1),p^*\xi/\mc O(-1))$, we can show that for a holomorphic vector bundle $\xi$ the stably complex structure~\eqref{eq:stably-complex-standard} is (stably) equivalent to the canonical complex structure on the complex manifold $\P(\xi)$. However, in what follows, we will be interested in \emph{non-standard} stably complex structures on complex projectivisations $\P(\xi)$.

\begin{definition}[Non-standard stably complex structure on $\P(\zeta\oplus \C)$]\label{def:stably-complex-nonstandard}
Consider a split vector bundle $\xi=\zeta\oplus\C$ over a base $B$. Define a \emph{non-standard} stably complex structure on the manifold $\P(\zeta\oplus\C)$ via an isomorphism of real vector bundles:
\begin{equation}\label{eq:stably-complex-nonstandard}
T\P(\zeta\oplus\C)\oplus\C\simeq (p^*\zeta\otimes \gamma)\oplus\gamma^*\oplus p^*TB.
\end{equation}
Slightly abusing notations, we denote the manifold $\P(\zeta\oplus\C)$ equipped with this stably complex structure by $\P(\zeta\oplus\overline{\C})$.
\end{definition}

The difference between the non-standard and standard stably complex structures is in the use of $\gamma^*$ instead of $\gamma$ in the formula above. A fiber of the stably complex bundle $\P(\zeta\oplus\overline{\C})$ is a projective space $\C P^{\dim\zeta}$ with a non-standard stably complex structure given by an isomorphism of real vector bundles:
\begin{equation}\label{eq:stably-complex-nonstandard-cp}
T\C P^{\dim\zeta}\oplus\C\simeq \gamma^{\oplus (\dim\zeta)}\oplus\gamma^*.
\end{equation}

\subsection{Blow-ups of complex submanifolds}\label{subsec:cmplxblowup}

Consider a smooth compact complex manifold $X$ and its complex submanifold $Z\subset X$. Denote by $Bl_{Z}X$ the blow-up of $X$ along $Z$.
The blow-up is a \emph{local} operation, i.e., it depends only on the tubular neighbourhood of the submanifold $Z\subset X$. Hence it is reasonable to expect that the difference $[Bl_{Z}X] - [X]$ is determined by the normal vector bundle $\nu(Z\subset X)$.

\begin{pr}[{Hitchin \cite[\S 4.5]{hi-74}}]\label{prop:hit}
Let $X$ and $Z$, with $Z\subset X$, be smooth compact complex manifolds of dimensions $n$ and $k$, resp. Consider a blow-up $\pi: Bl_{Z}X\rightarrow X$ along $Z$. Then the difference of classes of manifolds $Bl_{Z}X$ and $X$ in the unitary cobordism ring is:
\begin{equation}\label{eq:cobordism_blowup_identity}
[Bl_{Z}X] - [X] = -[\mathbb{P}(\nu(Z\subset X)\oplus\overline{\mathbb{C}})],
\end{equation}
where $\nu(Z\subset X)$ is a normal bundle to $Z$, and the projectivisation $\mathbb{P}(\nu(Z\subset X)\oplus\overline{\mathbb{C}})$ is equipped with the non-standard stably complex structure~\eqref{eq:stably-complex-nonstandard}.
\end{pr}

\begin{ex}[Blow-up at a point $Bl_x{X}\to X$]
Apply Proposition~\ref{prop:hit} to the blow-up $Bl_x{X}$ of a manifold $X$ at a point $x\in X$. In this case the normal bundle is trivial $\nu=\C^{\dim_\C X}$. So the formula~\eqref{eq:cobordism_blowup_identity} reduces to:
\begin{equation}\label{eq:cobordism_blowup_point}
[Bl_x{X}]-[X] = - [\P(\C^{\dim_\C X}\oplus\overline\C)].
\end{equation}
\end{ex}

It is well-known that the blow-up $\pi\colon Bl_{Z}X\to X$ of a manifold $X$ is obtained from $X$ by adding an exceptional divisor $E=\pi^{-1}(Z)\simeq\mathbb{P}(\nu(Z\subset X))$ with a normal bundle $\gamma=\mc O(-1)$ (cf.\,\cite[\S 6]{gr-ha-78}). In particular, if $\pi\colon Bl_x X\to X$ is a blow-up at a point $x\in X$ of some $n$-dimensional complex manifold, then the exceptional divisor $E=\pi^{-1}(x)\simeq \C P^{n-1}$ has normal bundle $\nu(E\subset Bl_x{X})\simeq \mc O(-1)$.

Now we define our key tool~--- a family of birational modifications $B_k(X)\to X$.
\begin{definition}[Modifications $B_k(X)$]\label{def:modification}
For any smooth complex manifold $X$ of complex dimension $n$ consider a blow-up $\pi\colon Bl_x{X}\to X$ at a point $x\in X$. Fix a number $0\le k\le n-2$ and pick a projective subspace $Z^k\simeq \C P^k$ in the exceptional divisor $E=\pi^{-1}(x)\simeq \C P^{n-1}$. 
Define the $k$-modification $B_{k}(X)$ to be the blow-up $Bl_Z(Bl_x{X})$ of the manifold $Bl_x{X}$ along $Z$.
\end{definition}

\begin{rem}
The notation $B_k(X)$ is ambiguous, since it does not specify the blown-up point $x\in X$ and submanifold $Z\subset E$, while different choices of $x\in X$ and $Z\subset E$ result into different complex manifolds $B_k(X)$. However, in this paper we are interested mainly in the cobordism class of $[B_k(X)]$, and by Proposition~\ref{prop:hit} the latter does not depend on the choices of $x$ and $Z$. 
\end{rem}

The normal bundle $\nu(Z\subset Bl_x{X})$ is isomorphic to $\mc O(-1)\oplus\mc O(1)^{\oplus(n-k-1)}$. Proposition~\ref{prop:hit} gives the following formula for the difference of cobordism classes $[B_k(X)]$ and $[Bl_x{X}]$:
\[
[B_k(X)]-[Bl_x{X}]=-[\P(\mc O(-1)\oplus\mc O(1)^{\oplus(n-k-1)}\oplus\overline{\C})].
\]

The stably complex manifold $D_{k,n}:=\P(\mc O(-1)\oplus\mc O(1)^{\oplus(n-k-1)}\oplus\overline{\C})$ is the projectivisation of an $(n-k+1)$-dimensional vector bundle $\xi=\mc O(-1)\oplus\mc O(1)^{\oplus(n-k-1)}\oplus\C$ over $Z\simeq\C P^k$. It follows from the definition of $D_{k,n}$ that
\begin{equation}\label{eq:sn_diff}
s_n(B_k(X))-s_n(Bl_x{X})=-s_n(D_{k,n}).
\end{equation}
In Section~\ref{sec:compute_sn} we compute the Milnor numbers $s_n(D_{k,n})$.

\subsection{Equivariant modifications \texorpdfstring{$B_{k}$}{Bk} of toric varieties}\label{sec:equivariant_modification}
We have defined the family of modifications $B_k(X)\to X$, $k=0,\dots,n-2$ for any compact complex manifold $X$. Now we describe the equivariant analogues of the operations $B_k(X)$ in the category of smooth projective toric varieties.

\begin{definition}
A smooth projective complex manifold $X$ is called (projective) \emph{toric}, if it admits an effective action of an algebraic torus $(\C^*)^{\dim_\C X}$ with an open and dense orbit.
\end{definition}

Any $n$-dimensional projective toric variety is uniquely determined by its underlying Delzant polytope $P$. The polytope $P$ is the image of the moment map $\mu\colon X\to \R^n$ for the action of a compact torus $U(1)^n\subset (\C^*)^n$ (for details cf.\,\cite[Chapter 5]{bu-pa-15}).

The modification $B_k(X)$ is determined by the choice of the point $x\in X$ and the $k$-dimensional projective subspace $Z=\C P^k$ in the exceptional divisor $E\subset Bl_x{X}$. Suppose that $X=X_P$ is a projective toric variety. It is well-known that the blow-up of $X$ at a fixed point~$x$ is an equivariant modification $Bl_x X\to X$. The underlying polytope for the manifold $Bl_x X$ is ${\rm cut}_p P$, i.e., the polytope obtained from $P$ by truncation of a vertex $p$ corresponding to the fixed point $x\in X$. The exceptional divisor $E\subset Bl_x X$ is a $(\C^*)^n$-invariant submanifold in $Bl_x X$. Similarly to the blow-up at a point, the blow-up $B_k(X)=Bl_Z(Bl_x X)\to X$ along any $(\C^*)^n$-invariant submanifold $Z\simeq \C P^k$, $Z\subset E\subset Bl_x(X)$ is equivariant. The manifold $B_k(X)$ is toric and corresponds to a polytope $P'_{k}={\rm cut}_{S^k}({\rm cut}_p P)$. This polytope is obtained from $P$ by two successive truncations (of a vertex $p$ and of one of the new $k$-dimensional faces $S^k$).

So we have proved the following:

\begin{pr}\label{st:bk-toric}
In the category of toric varieties the variety $B_k(X)=Bl_Z(Bl_x X)$ is toric and the modification $B_k(X)\to X$ is equivariant, provided the point $x\in X$ and the submanifold $Z\subset Bl_x X$ are invariant under the action of the algebraic torus $(\C^*)^{\dim X}$.
\end{pr}

From now on we consider only equivariant blow-ups and modifications of toric varieties.

\section{Construction of toric polynomial generators of the ring \texorpdfstring{$\Omega^U_*$}{OU}}

\subsection{Proof of main Theorem}

In this Section for every even dimension $n$ such that $n+1$ is not a power of a prime we construct a toric manifold of complex dimension $n$ with Milnor number $s_n(X)=1$. These are precisely the dimensions not covered by the results of~\cite{wi-13}. 
Let us define
\[
s_{k,n}=s_n(B_k(X))-s_n(X).	
\]
It follows from equations~\eqref{eq:cobordism_blowup_identity} and~\eqref{eq:cobordism_blowup_point} that the numbers $s_{k,n}$ do not depend on the choice of the manifold~$X$. We will need the following key lemma to prove the existence of toric polynomial generators of the unitary cobordism ring:

\begin{lm}\label{lm:key}
Let $n$ be an even number, such that $n+1$ is not a power of a prime. Then
\[
\mathrm{gcd}(s_{0,n},s_{1,n},\dots,s_{n-2,n})=1.
\]
\end{lm}

Prior to proving this lemma we use it to derive our main theorem:

\begin{thm}\label{thm:main}
There exists a sequence of smooth projective toric varieties $\{X^n\}_{n=1}^\infty$, representing polynomial generators of $\Omega^{U}_{*}$.
\end{thm}

The proof of this Theorem is based on several lemmata.

\begin{lm}\label{lm:ntheory}
Consider integers $t_{0},\dots,t_{l}$ such that $\mathrm{gcd}(t_0,\dots,t_l)=1$. Suppose that $t_0>0$. Then there exists a natural number $N=N(t_{0},\dots,t_{l})$ such that for any integer $x>N$ there exists a representation $x=\sum_{i=0}^{l}a_{i}t_{i}$, where $a_{i}\ge 0$, $i=0,\dots,l$ are non-negative integers.
\end{lm}

\begin{proof}
In case all $t_i$ are non-negative this is an elementary fact of Number Theory, related to a classical Frobenius problem, cf., for example,~\cite{br-42}.

The general case (i.e., some $t_i$ is negative) will be reduced to this particular one. First, for the absolute values $\{|t_i|\}_{i=0}^l$ construct a number $N=N(|t_0|,|t_1|,\dots,|t_l|)$ satisfying the conditions of the lemma. We claim that any integer $x>N$ can be represented as a linear combination of the numbers $\{t_i\}_{i=0}^l$ with non-negative integer coefficients. Indeed, let $x>N$. By the definition of $N$ there exists a linear combination
\[
x=\sum_{i=0}^l a_i \cdot |t_i|=\sum_{i=0}^l (a_i \cdot \sgn(t_i)) \cdot t_i
\]
with non-negative integer coefficients $a_i$, where $\sgn(t)$ denotes the sign of a number $t$. Recall that, by hypothesis, $t_0>0$. Hence, $\sgn(t_0)=1$. For any index $j$, such that $t_j$ is negative, replace $a_0$ with $a_0'=a_0-k t_j$, and $a_j\cdot \sgn(t_j)$ with $a_j'=a_j\cdot\sgn(t_j)+kt_0$, where $k\in \N$. Obviously the linear combination does not change with this replacement. For $k$ large enough ($k>a_j/t_0$) the coefficients of both $t_0$ and $t_j$ become positive. After performing this procedure for all negative $t_j$ one obtains the desired representation.
\end{proof}

\begin{lm}\label{lm:large_sn}
For any $N\in \N$ there exists an $n$-dimensional smooth projective toric variety~$X$ with Milnor number $s_n(X)>N$.
\end{lm}
\begin{proof}
Consider a vector bundle $\xi=\pi_1^*\mc O(-1)\oplus \pi_2^*\mc O(a)\oplus \C^{n-3}$ over $\C P^1\times \C P^1$, where $\pi_1,\pi_2\colon \C P^1\times \C P^1\to \C P^1$ are projections on the first and the second factors respectively. It follows from the computations of~\cite[\S 2.3]{sc-14} that the Milnor number of the projectivisation $\P(\xi)$ equals $s_n(\P(\xi))=(n+1)a$. Hence, for $a>N/(n+1)$ one has $s_n(\P(\xi))>N$. It remains to notice that the manifold $\P(\xi)$ is toric because
the total space of the vector bundle $\xi$ admits an action of the algebraic torus $(\C^{*})^{n+1}$ which descends to the action of the quotient torus $(\C^*)^n$ on $\P(\xi)$.
\end{proof}

\begin{proof}[Proof of Theorem~\ref{thm:main}]
As we mentioned at the beginning of this Section, to prove the Theorem it is enough to construct toric manifolds with Milnor number $1$ in all even dimensions $n$ such that $n+1$ is not a prime power. Let us fix such $n$.

According to Lemma~\ref{lm:key}, $\mathrm{gcd}(s_{0,n},\dots,s_{n-2,n})=1$. From the computations below (Formula~\eqref{eq:sn-blowup-pt}) it follows that $s_{0,n}<0$. We apply Lemma~\ref{lm:ntheory} to the numbers $-s_{0,n},\dots,-s_{n-2,n}$ and find $N=N(-s_{0,n},\dots,-s_{n-2,n})$ such that any integer $m>N$ can be represented as a linear combination of numbers $-s_{0,n},\dots,-s_{n-2,n}$ with some non-negative integer coefficients $a_0,\dots,a_{n-2}$.

Now take a toric manifold $X$ of dimension $n$ with Milnor number $s_n(X)>N+1$ and express the number $s_n(X)-1$ as an integer linear combination
\[
s_n(X)-1=-\sum_{i=0}^{n-2} a_i\cdot s_{i,n},
\]
with all $a_i$ non-negative.

Then subsequently apply $a_i$ (equivariant) modifications $B_i$ for every $i=0,\dots,n-2$ starting with the manifold~$X$. As a result, we obtain an $n$-dimensional toric variety $Y$ with the Milnor number
\[
s_n(Y)=s_n(X)+\sum_{i=0}^{n-2} a_i\cdot s_{i,n}=1.
\]
\end{proof}
\begin{rem}
If one starts with a \emph{projective} toric manifold $X$, then the construction in the proof leads to a projective manifold $Y$ as well. If $X$ is the projectivisation of the bundle $\xi=\pi_1^*\mc O(1)\oplus \pi_2^*\mc O(a)\oplus \C^{n-3}$ over $\C P^1\times \C P^1$, as in Lemma~\ref{lm:large_sn}, then the corresponding moment polytope is combinatorially equivalent to the product of simplices $P=\Delta^1\times \Delta^1\times\Delta^{n-2}$. In this case the polytope corresponding to the variety $Y$ is obtained from $P$ by a successive vertex and simplicial face truncations, according to the description of the equivariant modifications in Section~\ref{sec:equivariant_modification}.
\end{rem}

\subsection{Computation of Milnor numbers \texorpdfstring{$s_{n}(D_{k,n})$}{sndk}}\label{sec:compute_sn}

This and the following Subsection are devoted to the proof of Lemma~\ref{lm:key}, which we use in the proof of our main theorem.

First of all, we deduce an explicit formula for the Milnor numbers $s_n(D_{k,n})$ of the stably complex manifolds $D_{k,n}=\P(\mc O(-1)\oplus\mc O(1)^{\oplus(n-k-1)}\oplus\overline{\C})$. Similar computations were done in~\cite[\S 2.3]{sc-14}. Consider the natural projection $p\colon D_{k,n}\to \C P^k$. Denote by $u$ the positive generator of $H^2(\C P^k, \Z)$, i.e., $u=c_1(\mc O(1))$. Let $v$ be $c_1(\gamma)\in H^2(D_{k,n},\Z)$, where $\gamma=\mc O(1)$ is the ample line bundle along the fibers of the projection $p$. According to the isomorphism~\eqref{eq:leray-hirsch}, the cohomology ring $H^*(D_{k,n},\Z)$ is generated by $u$ and $v$:

\begin{equation}\label{cohom}
H^{*}(D_{k,n},\Z)\cong\mathbb{Z}[u,v]/(u^{k+1},v(v+u)^{n-k-1}(v-u)).
\end{equation}

It follows from the definition of the non-standard stably complex structure~\eqref{eq:stably-complex-nonstandard} on the manifold
$D_{k,n}:=\P(\mc O(-1)\oplus\mc O(1)^{\oplus(n-k-1)}\oplus\overline{\C})$ that
\[
c(T D_{k,n})=(1+u+v)^{n-k-1}(1-u+v)(1-v)\prod_{s=1}^k (1+w_s),
\]
where $w_1,\dots,w_k$ are Chern roots of $T\C P^k$. Since $w_s^n=0$ for all $s=1,\dots,k$, one gets
\begin{equation}\label{eq:sn}
s_n(D_{k,n})=\langle(n-k-1)(u+v)^n+(-u+v)^n+(-v)^n, [D_{k,n}]\rangle.
\end{equation}

We will need some auxillary lemmata to complete the computation of $s_n(D_{k,n})$.
\begin{lm}\label{lm:comp1}
For any $n, 0\le k\le n-2$ one has
\begin{equation}\label{eq:comp1}
\langle v^n, [D_{k,n}]\rangle=\sum_{i=0}^k (-1)^i 2^{k-i}\binom{n-1}{i}.
\end{equation}
\end{lm}
\begin{proof}
According to formula~\eqref{eq:integration},
\[
\langle v^n, [D_{k,n}]\rangle= \langle c^{-1}(\xi), [\C P^k]\rangle.
\]
The total Segre class of the bundle $\xi=\mc O(-1)\oplus\mc O(1)^{\oplus(n-k-1)}\oplus\C$ over $\C P^k$ is equal to $c^{-1}(\xi)=(1+u)^{-(n-k-1)}(1-u)^{-1}$. Hence,
\begin{multline}
\langle c^{-1}(\xi), [\C P^k]\rangle = \langle (1+u)^{-(n-k-1)}(1-u)^{-1}, [\C P^k]\rangle = \langle (1+u)^{-(n-k)}\biggl(1-\frac{2u}{1+u}\biggr)^{-1}, [\C P^k]\rangle=\\
=\langle\sum_{i=0}^\infty 2^i u^i (1+u)^{-(n-k+i)}, [\C P^k]\rangle=\sum_{i=0}^k 2^i (-1)^{k-i}\binom{n-1}{k-i},
\end{multline}
where in the last equality we use the fact that the coefficient of $u^{k-i}$ in a series $(1+u)^{-(n-k+i)}$ is equal to $(-1)^{k-i}\binom{n-1}{k-i}$.
\end{proof}

\begin{lm}\label{lm:comp2}
For any $n, 0\le k\le n-2$ one has
\begin{equation}
\langle (u+v)^n, [D_{k,n}]\rangle=2^{k+1}-1.
\end{equation}
\end{lm}
\begin{proof}
We use formula~\eqref{eq:integration} again, and substitute the expression of the Segre class 
\[c^{-1}(\xi)=(1+u)^{-(n-k-1)}(1-u)^{-1}\]
to obtain
\begin{multline*}
\langle (u+v)^n, [D_{k,n}]\rangle=\langle\sum_{i=0}^n \binom{n}{i}u^iv^{n-i}, [D_{k,n}]\rangle=\langle \sum_{i=0}^n \binom{n}{i}u^i c^{-1}(\xi), [\C P^k]\rangle=\\=
\langle (1+u)^n c^{-1}(\xi), [\C P^k]\rangle=\langle(1+u)^{k+1} (1+u+u^2+\cdots), [\C P^k]\rangle=\langle \sum_{i=0}^{k} \binom{k+1}{i}u^k, [\C P^k]\rangle=2^{k+1}-1.
\end{multline*}
\end{proof}

\begin{lm}\label{lm:comp3}
For any $n, 0\le k\le n-2$ one has
\begin{equation}
\langle (-u + v)^n, [D_{k,n}]\rangle=\sum_{i=0}^k (-2)^i\binom{n-1}{i}.
\end{equation}
\end{lm}
\begin{proof}
Similarly to the previous lemmata:
\begin{multline*}
\langle (-u+v)^n, [D_{k,n}]\rangle=\langle \sum_{i=0}^n (-1)^i\binom{n}{i}u^iv^{n-i}, [D_{k,n}]\rangle=\langle \sum_{i=0}^n (-1)^i\binom{n}{i}u^i c^{-1}(\xi), [\C P^k]\rangle=\\
=\langle (1-u)^{n-1} (1+u)^{-(n-k-1)}, [\C P^k]\rangle=
\langle ((1+u)-2u)^{n-1} (1+u)^{-(n-k-1)}, [\C P^k]\rangle=\\=\langle \sum_{i=0}^{(n-1)}(-2^i)\binom{n-1}{i} u^i (1+u)^{n-1-i}(1+u)^{-(n-k-1)}, [\C P^k]\rangle=\sum_{i=0}^k (-2)^{i}\binom{n-1}{i}.
\end{multline*}
\end{proof}

Putting together the expressions given by Lemmata~\ref{lm:comp1}, \ref{lm:comp2}, \ref{lm:comp3} and using the formula~\eqref{eq:sn} we obtain:

\begin{pr}\label{main}
The Milnor number of manifold $D_{k,n}$ satisfies
\begin{equation}\label{comp8}
s_{n}(D_{k,n})=(n-k-1)(2^{k+1}-1)+\sum_{i=0}^{k}(-1)^i\biggl(2^{i}+(-1)^{n}2^{k-i}\biggr)\binom{n-1}{i}.
\end{equation}
\end{pr}

\begin{ex}\label{ex:sn-dnk}
In the particular case $k=0$ we obtain a formula for the change of the Milnor number under a blow-up at a point, see formula~\eqref{eq:cobordism_blowup_point}:
\begin{equation}\label{eq:sn-blowup-pt}
s_{n}(Bl_{x}X)-s_n(X)=-s_n(D_{0,n})=-(n+(-1)^n),
\end{equation}
For $k=1$ and $k=n-2$ we have:
\begin{align}\label{eq:sn-dnk-examples}
\begin{split}
s_n(D_{1,n})=&\begin{cases}
0,\ &\mbox{for $n$ even},\\
2(n-3),\ &\mbox{for $n$ odd};
\end{cases}\\
s_n(D_{{n-2},n})=&\begin{cases}
2^n-1,\ &\mbox{for $n$ even},\\
0,\ &\mbox{for $n$ odd.}
\end{cases}
\end{split}
\end{align}
\end{ex}

\begin{cor}
For any compact complex manifold $X^n$, the change of the Milnor number under a sequential blow-up $B_k(X)=Bl_{Z^k}(Bl_x X)$ satisfies
\begin{multline}\label{eq:skn}
s_{k,n}=s_n(B_k(X))-s_n(X)=\\=
\bigl(s_n(B_k(X))-s_n(Bl_x X)\bigr)+\bigl(s_n(Bl_x X)-s_n(X)\bigr)=-s_{n}(D_{k,n})-(n+(-1)^{n}) = \\ =-\biggl((n-k-1)(2^{k+1}-1)+\sum_{i=0}^{k}(-1)^i\biggl(2^{i}+(-1)^n 2^{k-i}\biggr)\binom{n-1}{i}+n+(-1)^{n}\biggr),
\end{multline}
where $Z^k\simeq \C P^k\subset E$ is a projective subspace in the exceptional divisor of the blow-up $Bl_x X\to X$.
\end{cor}

\subsection{Changes of Milnor number \texorpdfstring{$s_{k,n}$}{snk} are coprime}
In the previous Subsection we deduced the formula~\eqref{comp8} for Milnor numbers $s_n(D_{k,n})$. Using this formula we computed the change of the Milnor numbers under the modifications $B_k(X)\to X$, see~\eqref{eq:skn}.

In this Subsection we prove Lemma~\ref{lm:key}, that is, we show that the numbers $s_{0,n},s_{1,n},\dots,s_{n-2,n}$ are coprime, provided $n$ is even and $n+1$ is not a prime power. First note that $s_{1,n}=-(n+1)$ (see Example~\ref{ex:sn-dnk}). To complete the proof we show that for any prime divisor $p$ of $n+1$ a certain integer linear combination of $s_{k,n}$ is not divisible by $p$.

Let us introduce a family of linear combinations of the numbers $s_{k,n}$ given by a more compact formula than~\eqref{eq:skn}.
\begin{pr}\label{mainnumf}
For any $n\geq 2$ and $k=2,\dots,n-2$ let $L_{k,n}=-s_{k,n}+3s_{k-1,n}-2s_{k-2,n}$. Then
\[
L_{k,n}=-2^{k}-1+(-1)^{n+k}\binom{n}{k}+(-2)^{k}\binom{n}{k}.
\]
In particular, for even $n$:
\[
L_{k,n}=-(2^{k}+1)\Bigl(1+(-1)^{k+1}\binom{n}{k}\Bigr).
\]
\end{pr}

\begin{proof}
A direct computation for $k=1,\dots,n-2$ involving the formula \eqref{eq:skn} yields:
$$
\begin{gathered}
-s_{k,n}+2s_{k-1,n}=-2^{k+1}-k+\sum_{i=0}^{k}(-2)^{i}\binom{n}{i}+(-1)^{n+k}\binom{n-1}{k}+1-(-1)^n.
\end{gathered}
$$
Applying this identity twice:
$$
\begin{gathered}
-s_{k,n}+3s_{k-1,n}-2s_{k-2,n}= (-s_{k,n}+2s_{k-1,n})-(-s_{k-1,n}+2s_{k-2,n})
\end{gathered}
$$
we get the desired formula.
\end{proof}

We will also need Lucas' Theorem, see~{\cite{fi-47}} for the proof.
\begin{thm}[Lucas]\label{lucas}
Let $p$ be prime, and let
$$
\begin{gathered}
n=n_{0}+n_{1}p+\dots+n_{r-1}p^{r-1}+n_{r}p^{r}\\
m=m_{0}+m_{1}p+\dots+m_{r-1}p^{r-1}+m_{r}p^{r}
\end{gathered}
$$
be the base $p$ expansions of the positive integers $n$ and $m$. Then one has
$$
\binom{n}{m}\equiv\binom{n_{0}}{m_{0}}\binom{n_{1}}{m_{1}}\dots\binom{n_{r}}{m_{r}}\pmod{p}.
$$
\end{thm}

Now we can carry out the proof of Lemma~\ref{lm:key}.

\begin{rem}
The assumptions on $n$ in Lemma~\ref{lm:key} are important. For any $n=p^m-1$ and prime $p$ all $s_{k,n}$ are divisible by $p$ and clearly not coprime. This follows from the properties of Milnor numbers: the additivity, vanishing on decomposables in $\Omega^{U}_{*}$ (cf.\,\cite{st-68}) and Theorem \ref{mn}. Therefore, the numbers $L_{k,n}$ are also divisible by $p$. To demonstrate these divisibility properties we list the values of $L_{k,n}$ for some $n$ and $k=2,\dots,n-2$: 
\begin{itemize}
\item[] $n=4=5-1:\ 25$;
\item[] $n=6=7-1:\ 70, -189, 238$;
\item[] $n=8=3^2-1:\ 135, -513, 1173, -1881, 1755$.
\end{itemize}
\end{rem}

\begin{proof}[Proof of Lemma~\ref{lm:key}]
We recall that $s_{1,n}=-(n+1).$ Therefore it suffice to show that for any prime divisor $p$ of $n+1$ some linear combination of $s_{0,n},\dots,s_{n-2,n}$ is not divisible by $p$. Using Proposition \ref{mainnumf} we will look for such a combination among the numbers $L_{k,n}=-(1+2^{k})(1+(-1)^{k+1}\binom{n}{k}),\ k=2,\dots,n-2$. Now we will find a number $k$ such that neither of the congruences
\begin{equation}\label{1eq}
2^{k} \equiv  -1\pmod{p}
\end{equation}
\begin{equation}\label{2eq}
\binom{n}{k} \equiv (-1)^{k}\pmod{p}
\end{equation}
is satisfied.
Consider the base $p$ expansion $n=n_{0}+n_{1}p+\dots+n_{r}p^{r}$ of $n$, where $0\le n_i\le (p-1)$, $n_r\neq 0$.
Notice, that $n_0=p-1$ (because the number $n+1$ is divisible by $p$). Moreover, there exists an index $j\le r$ such that $n_{j}<p-1$, because the number $n+1$ is not a prime power. Let $j$ be such a minimal index. Consider the number $k=p^j$. According to Lucas' Theorem,
$$
\binom{n}{k}\equiv\binom{n_j}{1}\equiv n_j\pmod{p}
$$
The congruence~\eqref{2eq} for this $k$ is not satisfied, since $k$ is odd and $0\le n_j<p-1.$ Consider two cases:

\emph{First Case.} Congruence~\eqref{1eq} is not satisfied as well. Then $k=p^j$ is the desired number.

\emph{Second Case.} Congruence~\eqref{1eq} is satisfied: $2^k\equiv-1\pmod{p}$. Let us replace $k$ with $k'=k+1$. By Lucas' Theorem the binomial coefficient $\binom{n}{k'}$ is congruent to $\binom{n_j}{1}\binom{p-1}{1}\equiv -n_{j}\pmod{p}$ and is not equal to $(-1)^{k'}=1,$ because $n_{j}<p-1$. Also we have $2^{k'}\equiv -2\not\equiv -1\pmod{p}$.

In both cases for at least one of the numbers $k\in \{p^j, p^j+1\}$ the combination $L_{k,n}=-s_{k,n}+3s_{k-1,n}-2s_{k-2,n}$ is not divisible by $p$. We have $j\ge 1$, hence $k\ge p>2$. It remains to make sure that $k\le n-2$. In the base $p$ expansion of $n$ one has $n_0=p-1$ and $n_r\ge 1$.
Therefore $n\ge p^r+(p-1)\ge p^j+(p-1)\ge k + (p-2)$. If $p>3$ or $r>j$, then $k\le n-2$, as required. Otherwise if $j=r$ and $p=3$, then $n= 3^r+\sum_{i=0}^{r-1} 2\cdot 3^i$ since $j$ is the minimal number such that $n_j\neq p-1=2$. If $r>1$, then $k=3^r+1<n-2$, as required. The only case left to check is $p=3$, $j=r=1$. Then $n=5$, while in Lemma~\ref{lm:key} we assume $n$ is even.
\end{proof}

\section{Properties of modifications \texorpdfstring{$B_k$}{Bk} and their applications}
\subsection{Modifications \texorpdfstring{$B_k$}{Bk} and polytope operations}
In this Subsection we study a connection between equivariant modifications $B_k$ of toric manifolds and the corresponding operations on Delzant polytopes.

Recall that for any smooth projective toric variety $X$, with the underlying $n$-dimensional polytope $P$, the toric variety $B_k(X)$ corresponds to the polytope ${\rm cut}_{S^k}({\rm cut}_p P )$ obtained from $P$ by successive vertex $p$ and $k$-dimensional face $S^k$ truncations. The choices of a vertex and a face correspond to the choices of a fixed point $x\in X$ and a submanifold $Z\subset E\subset Bl_x(X)$ to blow-up. It turns out that for a specific choice of submanifolds $Z_1=\C P^{k}$ and $ Z_2=\C P^{n-k-2}$ the modifications $B_{k}(X^{n})$ and $B_{n-k-2}(X^{n})$ of a toric variety $X^{n}$ lead to a pair of toric varieties with combinatorially equivalent moment polytopes.

\begin{pr}\label{equiv}
Consider a truncation of a simple polytope $P$ at a vertex $p$. Denote the new facet of the truncated polytope ${\rm cut}_{p}P$ by $G\subset {\rm cut}_{p}P$ ($G$ is the simplex $\Delta^{n-1}$). For any two complementary (i.e.,\,non-intersecting) faces $S_{1}=\Delta^{k}$ and $S_{2}=\Delta^{n-k-2}$ of the simplicial facet $G$, the polytopes $P'_{1}={\rm cut}_{S_{1}}({\rm cut}_{p}P)$ and $P'_{2}={\rm cut}_{S_{2}}({\rm cut}_{p}P)$ are combinatorially equivalent.
\end{pr}

\begin{proof}
Consider the sets of facets of the polytopes $P_{1}$ and $P_{2}$
\[
\mc F_1=\{F\,|\,F\in P_1\},\quad \mc F_2=\{F\,|\,F\in P_2\}.
\]
We will prove that there exists a bijection $f\colon \mc F_1\to \mc F_2$ such that the facets $F_{i_1},\dots,F_{i_k}$ of the polytope $P_1$ intersect iff the facets $f(F_{i_1}),\dots,f(F_{i_k})$ of the polytope $P_2$ intersect. Due to the simplicity of the polytopes $P_{1}$ and $P_{2},$ it suffices to check this property only for the maximal ($n$-fold) intersections $F_{i_1}\cap\dots\cap F_{i_n}$.

Consider the set $\mc F=\{F|F\in P\}$ of facets of the polytope $P$. The polytope $P_{i}$ has a new facet $G_{i}$, obtained by truncating the face $S_{i}$ ($i=1,2$). Then one has $\mc F_1=\mc F\cup\{G_1, G\}$ and $\mc F_2=\mc F\cup\{G_2, G\}.$

We define the map $f\colon \mc F_1\to \mc F_2$ as follows:
\[
\begin{cases}
f(G)=G_2,\\
f(G_1)=G,\\
f(F)=F,\mbox{ if } F\neq G,G_1.
\end{cases}
\]

We claim that the bijection $f$ induces a combinatorial isomorphism of the polytopes $P_1$ and $P_2$.

 Denote the facets intersecting in the vertex $p\in P$ by $H_1,\dots, H_n$. The simplices $S_1$ and $S_2$ are complementary by hypothesis, hence, (possibly after a relabelling of the facets $H_i$) $S_1=G\cap H_1\cap\dots H_{n-k-1}$, and $S_2=G\cap H_{n-k}\cap\dots\cap H_n$ in the polytope ${\rm cut}_p P$. The polytopes $P_{1}$ and $P_{2}$ are the same outside the neighbourhoods of facets $G, G_1$ and $G, G_2,$ respectively. Hence $f$ induces a combinatorial isomorphism $\mc F_1\backslash \{G,G_1\}\to \mc F_2\backslash \{G,G_2\}$. It remains to consider $n$-fold intersections $F_{i_1}\cap\dots\cap F_{i_n}$, which include one of the facets $G,G_1$ in $P_1$ ($G, G_2$ in $P_2$, respectively).

One can write down all the non-empty intersections $F_{i_1}\cap\dots\cap F_{i_n}$, which include at least one of the facets $G$ and $G_1$ of $P_1$:
\[
\begin{gathered}
G\cap H_{1}\cap\dots\cap \widehat{H_j}\cap\dots\cap H_{n},\ \mbox{where } 1\le j< n-k,\\
G_1\cap H_{1}\cap\dots\cap \widehat{H_j}\cap\dots\cap H_{n},\ \mbox{where } n-k\le j\le n,\\
G\cap G_1\cap H_{1}\cap\dots\cap \widehat{H_{j_1}}\cap\dots\cap \widehat{H_{j_2}}\cap\dots\cap H_{n},\ \mbox{where } 1\le j_1<n-k\le j_2\le n.\\
\end{gathered}
\]
Similarly, all the non-empty intersections $F_{i_1}\cap\dots\cap F_{i_n}$ involving at least one of the facets $G$ and $G_2$ of $P_2$ are:
\[
\begin{gathered}
G_2\cap H_{1}\cap\dots\cap \widehat{H_j}\cap\dots\cap H_{n},\ \mbox{where } 1\le j< n-k,\\
G\cap H_{1}\cap\dots\cap \widehat{H_j}\cap\dots\cap H_{n},\ \mbox{where } n-k\le j\le n,\\
G\cap G_2\cap H_{1}\cap\dots\cap \widehat{H_{j_1}}\cap\dots\cap \widehat{H_{j_2}}\cap\dots\cap H_{n},\ \mbox{where } 1\le j_1 < n-k\le j_2\le n.\\
\end{gathered}
\]

Clearly, bijections $f$ and $f^{-1}$ map $n$-fold intersections in one list to the $n$-fold intersections in the other. This observation and the fact that the bijection $f$ acts identically on all of the facets, except $G$ and $G_{1}$, prove our claim. Hence $f$ induces a combinatorial isomorphism between $P_1$ and $P_2$.
\end{proof}

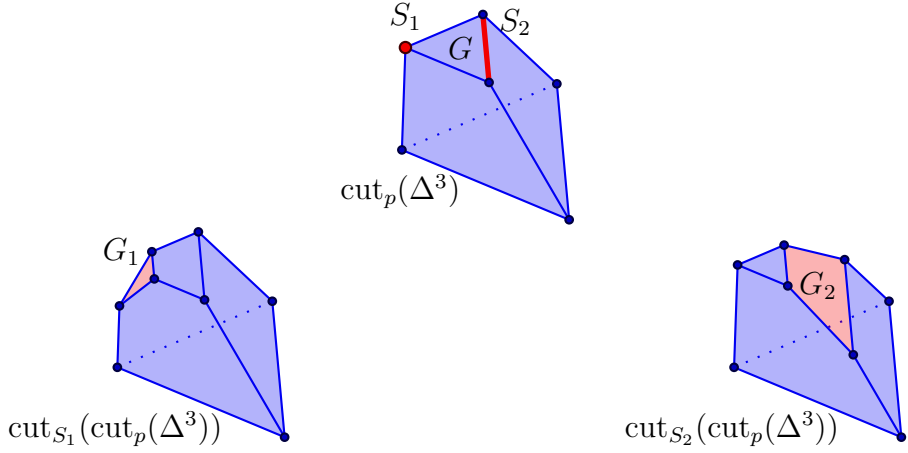
\begin{figure}
\centering
\caption{Polytopes, corresponding to the modifications $Bl_x \C P^3$, $B_0(\C P^3)$ and $B_1(\C P^3)$.\label{fig:cp3}}

\begin{tikzpicture}%
	[x={(0.733255cm, -0.308755cm)},
	y={(0.679249cm, 0.292077cm)},
	z={(0.030936cm, 0.905186cm)},
	scale=3.000000,
	back/.style={loosely dotted, thick},
	edge/.style={color=blue!95!black, thick},
	edgered/.style={color=red!95!black, line width=2pt},
	facet/.style={fill=blue!95!black,fill opacity=0.300000},
	facetlast/.style={fill=red!95!black,fill opacity=0.300000},
	vertex/.style={inner sep=1pt,circle,draw=blue!25!black,fill=blue!75!black,thick,anchor=base},
	vertexred/.style={inner sep=1.5pt,circle,draw=red!25!black,fill=red!95!black,thick,anchor=base}]

%
%
\coordinate (0.00000, 0.00000, 0.50000) at (0.00000, 0.00000, 0.50000);
\coordinate (0.00000, 0.00000, 0.00000) at (0.00000, 0.00000, 0.00000);
\coordinate (1.00000, 0.00000, 0.00000) at (1.00000, 0.00000, 0.00000);
\coordinate (0.00000, 1.00000, 0.00000) at (0.00000, 1.00000, 0.00000);
\coordinate (0.50000, 0.00000, 0.50000) at (0.50000, 0.00000, 0.50000);
\coordinate (0.00000, 0.50000, 0.50000) at (0.00000, 0.50000, 0.50000);
\draw[edge,back] (0.00000, 0.00000, 0.00000) -- (0.00000, 1.00000, 0.00000);
\fill[facet] (0.50000, 0.00000, 0.50000) -- (0.00000, 0.00000, 0.50000) -- (0.00000, 0.00000, 0.00000) -- (1.00000, 0.00000, 0.00000) -- cycle {};
\fill[facet] (0.00000, 0.50000, 0.50000) -- (0.00000, 1.00000, 0.00000) -- (1.00000, 0.00000, 0.00000) -- (0.50000, 0.00000, 0.50000) -- cycle {};
\fill[facet] (0.00000, 0.50000, 0.50000) -- (0.00000, 0.00000, 0.50000) -- (0.50000, 0.00000, 0.50000) -- cycle {};
\draw[edge] (0.00000, 0.00000, 0.50000) -- (0.00000, 0.00000, 0.00000);
\draw[edge] (0.00000, 0.00000, 0.50000) -- (0.50000, 0.00000, 0.50000);
\draw[edge] (0.00000, 0.00000, 0.50000) -- (0.00000, 0.50000, 0.50000);
\draw[edge] (0.00000, 0.00000, 0.00000) -- (1.00000, 0.00000, 0.00000);
\draw[edge] (1.00000, 0.00000, 0.00000) -- (0.00000, 1.00000, 0.00000);
\draw[edge] (1.00000, 0.00000, 0.00000) -- (0.50000, 0.00000, 0.50000);
\draw[edge] (0.00000, 1.00000, 0.00000) -- (0.00000, 0.50000, 0.50000);
\draw[edgered] (0.50000, 0.00000, 0.50000) -- node[above right] {\color{black}$S_2$} node[left] {\color{black}$G$} (0.00000, 0.50000, 0.50000);
\node[vertexred, label=above:{$S_1$}] at (0.00000, 0.00000, 0.50000)     {};
\node[vertex] at (0.00000, 0.00000, 0.00000)     {};
\node[vertex] at (1.00000, 0.00000, 0.00000)     {};
\node[vertex] at (0.00000, 1.00000, 0.00000)     {};
\node[vertex] at (0.50000, 0.00000, 0.50000)     {};
\node[vertex] at (0.00000, 0.50000, 0.50000)     {};

\node at (0.00, 0.00, -0.20) {${\rm cut}_p(\Delta^3) $};
\end{tikzpicture}

\begin{tikzpicture}%
	[x={(0.733255cm, -0.308755cm)},
	y={(0.679249cm, 0.292077cm)},
	z={(0.030936cm, 0.905186cm)},
	scale=3.000000,
	back/.style={loosely dotted, thick},
	edge/.style={color=blue!95!black, thick},
	facet/.style={fill=blue!95!black,fill opacity=0.300000},
	facetlast/.style={fill=red!95!black,fill opacity=0.300000},
	vertex/.style={inner sep=1pt,circle,draw=blue!25!black,fill=blue!75!black,thick,anchor=base}]	
%
%
\coordinate (0.00000, 0.20000, 0.50000) at (0.00000, 0.20000, 0.50000);
\coordinate (0.20000, 0.00000, 0.50000) at (0.20000, 0.00000, 0.50000);
\coordinate (0.00000, 0.00000, 0.00000) at (0.00000, 0.00000, 0.00000);
\coordinate (0.00000, 0.00000, 0.30000) at (0.00000, 0.00000, 0.30000);
\coordinate (0.00000, 1.00000, 0.00000) at (0.00000, 1.00000, 0.00000);
\coordinate (1.00000, 0.00000, 0.00000) at (1.00000, 0.00000, 0.00000);
\coordinate (0.50000, 0.00000, 0.50000) at (0.50000, 0.00000, 0.50000);
\coordinate (0.00000, 0.50000, 0.50000) at (0.00000, 0.50000, 0.50000);
\draw[edge,back] (0.00000, 0.00000, 0.00000) -- (0.00000, 1.00000, 0.00000);
\fill[facet] (0.50000, 0.00000, 0.50000) -- (0.20000, 0.00000, 0.50000) -- (0.00000, 0.00000, 0.30000) -- (0.00000, 0.00000, 0.00000) -- (1.00000, 0.00000, 0.00000) -- cycle {};
\fill[facet] (0.00000, 0.50000, 0.50000) -- (0.00000, 1.00000, 0.00000) -- (1.00000, 0.00000, 0.00000) -- (0.50000, 0.00000, 0.50000) -- cycle {};
\fill[facet] (0.00000, 0.50000, 0.50000) -- (0.00000, 0.20000, 0.50000) -- (0.20000, 0.00000, 0.50000) -- (0.50000, 0.00000, 0.50000) -- cycle {};
\fill[facetlast] (0.00000, 0.00000, 0.30000) -- (0.00000, 0.20000, 0.50000) -- (0.20000, 0.00000, 0.50000) -- cycle {};
\draw[edge] (0.00000, 0.20000, 0.50000) -- (0.20000, 0.00000, 0.50000);
\draw[edge] (0.00000, 0.20000, 0.50000) -- (0.00000, 0.00000, 0.30000);
\draw[edge] (0.00000, 0.20000, 0.50000)  node[left] {\color{black}$G_1$} -- (0.00000, 0.50000, 0.50000);
\draw[edge] (0.20000, 0.00000, 0.50000) -- (0.00000, 0.00000, 0.30000);
\draw[edge] (0.20000, 0.00000, 0.50000) -- (0.50000, 0.00000, 0.50000);
\draw[edge] (0.00000, 0.00000, 0.00000) -- (0.00000, 0.00000, 0.30000);
\draw[edge] (0.00000, 0.00000, 0.00000) -- (1.00000, 0.00000, 0.00000);
\draw[edge] (0.00000, 1.00000, 0.00000) -- (1.00000, 0.00000, 0.00000);
\draw[edge] (0.00000, 1.00000, 0.00000) -- (0.00000, 0.50000, 0.50000);
\draw[edge] (1.00000, 0.00000, 0.00000) -- (0.50000, 0.00000, 0.50000);
\draw[edge] (0.50000, 0.00000, 0.50000) -- (0.00000, 0.50000, 0.50000);
\node[vertex] at (0.00000, 0.20000, 0.50000)     {};
\node[vertex] at (0.20000, 0.00000, 0.50000)     {};
\node[vertex] at (0.00000, 0.00000, 0.00000)     {};
\node[vertex] at (0.00000, 0.00000, 0.30000)     {};
\node[vertex] at (0.00000, 1.00000, 0.00000)     {};
\node[vertex] at (1.00000, 0.00000, 0.00000)     {};
\node[vertex] at (0.50000, 0.00000, 0.50000)     {};
\node[vertex] at (0.00000, 0.50000, 0.50000)     {};

\node at (0.00, 0.00, -0.30) {${\rm cut}_{S_1}({\rm cut}_p(\Delta^3)) $};
\end{tikzpicture}\hspace{4cm}
\begin{tikzpicture}%
	[x={(0.733255cm, -0.308755cm)},
	y={(0.679249cm, 0.292077cm)},
	z={(0.030936cm, 0.905186cm)},
	scale=3.000000,
	back/.style={loosely dotted, thick},
	edge/.style={color=blue!95!black, thick},
	facet/.style={fill=blue!95!black,fill opacity=0.300000},
	facetlast/.style={fill=red!95!black,fill opacity=0.300000},	
	vertex/.style={inner sep=1pt,circle,draw=blue!25!black,fill=blue!75!black,thick,anchor=base}]	
	
%
\coordinate (0.00000, 0.00000, 0.50000) at (0.00000, 0.00000, 0.50000);
\coordinate (0.00000, 0.00000, 0.00000) at (0.00000, 0.00000, 0.00000);
\coordinate (1.00000, 0.00000, 0.00000) at (1.00000, 0.00000, 0.00000);
\coordinate (0.00000, 1.00000, 0.00000) at (0.00000, 1.00000, 0.00000);
\coordinate (0.00000, 0.30000, 0.50000) at (0.00000, 0.30000, 0.50000);
\coordinate (0.30000, 0.00000, 0.50000) at (0.30000, 0.00000, 0.50000);
\coordinate (0.70000, 0.00000, 0.30000) at (0.70000, 0.00000, 0.30000);
\coordinate (0.00000, 0.70000, 0.30000) at (0.00000, 0.70000, 0.30000);
\draw[edge,back] (0.00000, 0.00000, 0.00000) -- (0.00000, 1.00000, 0.00000);
\fill[facet] (0.70000, 0.00000, 0.30000) -- (1.00000, 0.00000, 0.00000) -- (0.00000, 0.00000, 0.00000) -- (0.00000, 0.00000, 0.50000) -- (0.30000, 0.00000, 0.50000) -- cycle {};
\fill[facet] (0.00000, 0.70000, 0.30000) -- (0.00000, 1.00000, 0.00000) -- (1.00000, 0.00000, 0.00000) -- (0.70000, 0.00000, 0.30000) -- cycle {};
\fill[facet] (0.30000, 0.00000, 0.50000) -- (0.00000, 0.00000, 0.50000) -- (0.00000, 0.30000, 0.50000) -- cycle {};
\fill[facetlast] (0.00000, 0.70000, 0.30000) -- (0.00000, 0.30000, 0.50000) -- (0.30000, 0.00000, 0.50000) -- (0.70000, 0.00000, 0.30000) -- cycle {};
\draw[edge] (0.00000, 0.00000, 0.50000) -- (0.00000, 0.00000, 0.00000);
\draw[edge] (0.00000, 0.00000, 0.50000) -- (0.00000, 0.30000, 0.50000);
\draw[edge] (0.00000, 0.00000, 0.50000) -- (0.30000, 0.00000, 0.50000)  node[right] {\color{black}$G_2$};
\draw[edge] (0.00000, 0.00000, 0.00000) -- (1.00000, 0.00000, 0.00000);
\draw[edge] (1.00000, 0.00000, 0.00000) -- (0.00000, 1.00000, 0.00000);
\draw[edge] (1.00000, 0.00000, 0.00000) -- (0.70000, 0.00000, 0.30000);
\draw[edge] (0.00000, 1.00000, 0.00000) -- (0.00000, 0.70000, 0.30000);
\draw[edge] (0.00000, 0.30000, 0.50000) -- (0.30000, 0.00000, 0.50000);
\draw[edge] (0.00000, 0.30000, 0.50000) -- (0.00000, 0.70000, 0.30000);
\draw[edge] (0.30000, 0.00000, 0.50000) -- (0.70000, 0.00000, 0.30000);
\draw[edge] (0.70000, 0.00000, 0.30000) -- (0.00000, 0.70000, 0.30000);
\node[vertex] at (0.00000, 0.00000, 0.50000)     {};
\node[vertex] at (0.00000, 0.00000, 0.00000)     {};
\node[vertex] at (1.00000, 0.00000, 0.00000)     {};
\node[vertex] at (0.00000, 1.00000, 0.00000)     {};
\node[vertex] at (0.00000, 0.30000, 0.50000)     {};
\node[vertex] at (0.30000, 0.00000, 0.50000)     {};
\node[vertex] at (0.70000, 0.00000, 0.30000)     {};
\node[vertex] at (0.00000, 0.70000, 0.30000)     {};

\node at (0.00, 0.00, -0.30) {${\rm cut}_{S_2}({\rm cut}_p(\Delta^3)) $};
\end{tikzpicture}
\end{figure}

\begin{rem}
It is important to assume that the simplices $S_{1}$ and $S_{2}$ are complementary in $G$, otherwise Proposition \ref{equiv} does not necessarily hold. 
\end{rem}

\begin{cor}\label{cor:comb-inv}
Consider an $n$-dimensional smooth projective toric variety $X$. Take a fixed point $x\in X$ and consider a blow-up $\pi\colon Bl_x X\to X$. Pick two non-intersecting $(\C^*)^n$-invariant subspaces $\C P^k\simeq Z_1\subset E\subset Bl_x X$ and $\C P^{n-k-2}\simeq Z_2\subset E\subset Bl_x X$ in the exceptional divisor of $\pi$. Then the blow-ups $B_k(X)=Bl_{Z_1}(Bl_x X)$ and $B_{n-k-2}(X)=Bl_{Z_2}(Bl_x X)$ are toric varieties with combinatorially equivalent moment polytopes.
\end{cor}

\begin{ex}
Consider the simplest three-dimensional compact toric variety $X=\C P^3$. Its underlying polytope is the three-dimensional simplex $P=\Delta^3$. The blow-up of $\C P^3$ at a $(\C^*)^3$-invariant point corresponds to the vertex truncation ${\rm cut }_{p}(\Delta^3)$ of the simplex $\Delta^3$.  The new facet of the vertex-truncated simplex ${\rm cut }_{p}(\Delta^3)$ is a two-dimensional simplex. Let us pick any vertex $S_{1}$ and the complementary edge $S_{2}$ of this new facet. The polytopes ${\rm cut}_{S^{0}}({\rm cut}_{p}(\Delta^3))$ and ${\rm cut}_{S^{1}}({\rm cut}_{p}(\Delta^3))$ are the moment polytopes for the toric varieties $B_0(X)$ and $B_1(X)$. Figure\,\ref{fig:cp3} demonstrates that these polytopes are indeed combinatorially equivalent.
%
%
\end{ex}

\subsection{Combinatorial rigidity of Hirzebruch genera}

In this Subsection we give an application of the modifications $B_k$ to the theory of Hirzebruch genera. More specifically, we give a \emph{combinatorial} characterisation of the two-parameter Todd genus.

Let $R$ be a commutative ring with a unit. \emph{Hirzebruch genus} with values in $R$ is a ring homomorphism $\phi:\ \Omega^{U}_{*}\to R$. (The term \emph{multiplicative genus} is also used). Assume that ring $R$ does not have additive torsion. Then any genus $\phi$ is uniquely determined by its extension $\phi_\mathbb Q\colon \Omega^U_*\otimes \mathbb Q\to R\otimes \mathbb Q$.
Every Hirzebruch genus $\phi:\ \Omega_{*}^U\rightarrow R$ corresponds to a formal power series
\[
Q(x)=1+\sum_{k=1}^{\infty}q_{k} x^{k}\in (R\otimes\Q)[[x]].
\]
 In this case the value of $\phi$ on a stably complex manifold $M^{2n}$ is given by the formula $\phi(M):=\langle  Q(t_1)\cdots Q(t_n), [M]\rangle$, where $t_i$ are Chern roots of $M$ (for more details on Hirzebruch genera see \cite{hi-66, hi-92}).

\begin{ex}[Two-parameter Todd genus]
The two-parameter Todd genus $\chi_{a,b}\colon \Omega^U_*\to \Z[a,b]$ is given by the $Q$-series $x\cdot \cfrac{ae^{ax}-be^{bx}} {e^{ax}-e^{bx}}$ (cf.\,the proof in \cite[Lemma 2.3]{kr-74}). Some specializations of the two-parameter Todd genus include: signature (for $a=1,\ b=-1$), arithmetic genus (for $a=1,\ b=0$), Euler characteristic (for $a=1$ and taking the limit as $b\to 1$).  With the specialization $a=1,\ b=-y$ the genus $\chi_{a,b}$ turns into the $\chi_{y}$-genus. Similarly to the $\chi_y$-genus (cf.\,\cite[\S 21]{hi-66},\cite[\S 5.4]{hi-92}), the value of $\chi_{a,b}$ on any $n$-dimensional complex manifold can be expressed in terms of the the dimensions of the cohomology groups of holomorphic differentials:
\begin{equation}\label{eq:todd}
\chi_{a,b}(M)=\sum_{i,j=0}^n \dim_\C H^j(M; \Omega^i) (-1)^{i+j}a^{n-i}b^{i}.
\end{equation}

Notice that the image of the two-parameter Todd genus is the ring of symmetric polynomials $\Z[\sigma_1,\sigma_2]\subset\Z[a,b]$ in two variables: $\sigma_1=a+b=\chi_{a,b}([\C P^1]), \sigma_2=ab=\chi_{a,b}([\C P^1]^2-[\C P^2])$. Further we consider the genus $\chi_{a,b}$ as a ring homomorphism onto $\Z[\sigma_1,\sigma_2]$.
\end{ex}

\begin{definition}
We say that a Hirzebruch genus $\phi\colon\Omega^U_*\to R$ is \emph{combinatorially rigid}, if for any two smooth projective toric varieties $X_1$ and $X_2$ with combinatorially equivalent moment polytopes $P_1\simeq P_2$ the corresponding values of $\phi$ coincide: $\phi(X_1)=\phi(X_2)$.
\end{definition}

\begin{ex}
The Euler characteristic of a toric variety $\chi(X_P)$ is equal to the number of vertices of the underlying moment polytope $P$. Hence, it is combinatorially rigid.
\end{ex}

One can deduce from the results of Danilov (cohomology of toric varieties, cf.\,\cite[\S 12]{da-78}) and formula~\eqref{eq:todd} that the two-parameter Todd genus $\chi_{a,b}$ is given by:
\[
\chi_{a,b}(X_P)=\sum_{i=0}^n h_i(P) a^i b^{n-i},
\]
where $h_i(P)$ are the components of the $h$-vector of the polytope $P$, cf.\,\cite[Chapter 1, \S 3]{bu-pa-15}. The numbers $h_i(P)$ are combinatorial invariants of the polytope $P$. Hence, the two-parameter Todd genus is combinatorially rigid.

It turns out that any combinatorially rigid Hirzebruch genus is a specialization of the two-parameter Todd genus.

\begin{thm}\label{thm:2todd_toric}
Let $\phi\colon \Omega^U_*\to R$ be a combinatorially rigid $R$-genus. Then there exists a unique ring homomorphism $f\colon \Z[\sigma_1,\sigma_2]\to R$ such that $\phi=f\circ\chi_{a,b}:$
\[
\begin{tikzcd}[column sep=2]
\Omega^U_* \arrow{dr}{\chi_{a,b}} \arrow{rr}{\phi} && R\\
 & \Z[\sigma_1,\sigma_2] \arrow[dashed]{ur}{f}
\end{tikzcd}
\]
\end{thm}
\begin{proof}
By our assumption the ring $R$ is torsion-free. We construct a homomorphism $f_\Q\colon \Q[\sigma_1,\sigma_2]\to R\otimes \Q$ such that $\phi\otimes \Q=f_\Q\circ(\chi_{a,b}\otimes\Q)$.

Consider the ideal $\mc I\subset \Omega^U_*\otimes \Q$ generated by all the differences $[X_1]-[X_2]$, where $X_1, X_2$ are smooth projective toric varieties with combinatorially equivalent moment polytopes. It follows from the definition of combinatorial rigidity that $\phi\otimes \Q$ factors through the quotient ring $(\Omega^U_*\otimes \Q)/\mc I$. Recall that the two-parameter Todd genus is combinatorially rigid and maps the ring $\Omega^U_*\otimes \Q$ onto $\Q[\sigma_1,\sigma_2]$. Therefore there exists an epimorphism $\pi_1\colon (\Omega^U_*\otimes \Q)/\mc I\to \Q[\sigma_1,\sigma_2]$ such that the composition
\[
\Omega^U_*\otimes \Q\to (\Omega^U_*\otimes \Q)/\mc I\xrightarrow{\pi_1} \Q[\sigma_1,\sigma_2]
\]
is equal to $\chi_{a,b}\otimes \Q$. Now it remains to check that the projection $\pi_1$ is an isomorphism.

For any $n\ge 3$ consider two equivariant modifications $B_0(\C P^n)$ and $B_{n-2}(\C P^n)$ with combinatorially equivalent moment polytopes (cf.\,Corollary~\ref{cor:comb-inv}). It follows from the computations of Example~\ref{ex:sn-dnk} that $s_n(B_0(\C P^n))\neq s_n(B_{n-2}(\C P^n))$. Hence the elements $[\C P^1], [\C P^2]$, $\{[B_0(\C P^n)]-[B_{n-2}(\C P^n)]\}_{n\ge 3}$ are polynomial generators of the ring $\Omega^U_*\otimes \Q$. By construction, the elements $\{[B_0(\C P^n)]-[B_{n-2}(\C P^n)]\}_{n\ge 3}$ belong to the ideal $\mc I$. Hence, there exists a natural projection $\pi_2\colon \Q[\C P^1,\C P^2]\to (\Omega^U_*\otimes \Q)/\mc I$. The composition
\[
\Q[\C P^1,\C P^2]\xrightarrow{\pi_2} (\Omega^U_*\otimes \Q)/\mc I\xrightarrow{\pi_1} \Q[\sigma_1,\sigma_2]
\]
coincides with the evaluation map of the genus $\chi_{a,b}$ and is an isomorphism: $[\C P^1]\mapsto \sigma_1$, $[\C P^2]\mapsto \sigma_1^2-\sigma_2$. Therefore $\pi_1$ is also an isomorphism.

We have obtained a factorisation of the homomorphism $\phi\otimes\Q\colon \Omega_*^U\otimes\Q\to R\otimes\Q$ through the two-parameter Todd genus $\chi_{a,b}\otimes \Q\colon \Omega^U_*\otimes\Q\to\Q[\sigma_1,\sigma_2]$:
\[
\begin{tikzcd}[column sep = 2]
\Omega^U_*\otimes\Q \arrow{dr}[below left]{\chi_{a,b}\otimes \Q} \arrow{rr}{\phi\otimes\Q} && R\otimes \Q \\
 & \Q[\sigma_1,\sigma_2] \arrow[dashed]{ur}[below right]{f_\Q}
\end{tikzcd}
\]
Now we define $f$ by restricting $f_\Q$ to $\Z[\sigma_1,\sigma_2]\subset\Q[\sigma_1,\sigma_2]$. The image of $f$ is in $R$, since $f_\Q(\sigma_1)=f_\Q\circ \chi_{a,b}(\C P^1)=(\phi\otimes \Q)(\C P^1)$ and the latter lies in $R$, as the image of $\Omega^U$ under $\phi$ is in $R$ (similar argument works for the image of $\sigma_2$). Hence $f$ maps $\Z[\sigma_1,\sigma_2]$ to $R$ and $\phi=f\circ\chi_{a,b}$ is the required factorisation.

Uniqueness of $f$ is obvious, since necessarily $f(\sigma_1)=\phi([\C P^1])$ and $f(\sigma_2)=\phi([\C P^1]^2-[\C P^2])$
\end{proof}

\section{Concluding remarks}
In this paper we have constructed polynomial generators of the unitary cobordism ring $\Omega^U_*$ among the smooth projective toric varieties. We note, however, that every manifold provided by Theorem~\ref{thm:main} is the result of a very large number of modifications $B_{k_i},$ thus has quite complicated topology. So it is interesting to seek toric polynomial generators of $\Omega^{U}_{*}$ with the ``smallest topology'':

{\medskip\noindent\bf Problem.} Find projective toric polynomial generators of the ring $\Omega^U_*$ with moment polytopes having:
\begin{itemize}
\item[a)] the least number of vertices;
\item[b)] the least number of facets;
\item[c)] the least number of faces of all dimensions.
\end{itemize}

Analysing the modifications $B_k$ we gave a characterization of the two-parameter Todd genus in terms of combinatorial rigidity (cf.\,Theorem~\ref{thm:2todd_toric}). Combinatorial rigidity of any genus $\phi$ requires the coincidence of the values $\phi(X_1)=\phi(X_2)$ for any pair of toric varieties with combinatorially equivalent moment polytopes. It is natural to generalise the notion of combinatorial rigidity in the following way:
\begin{definition}
Let $\mc P = \{P_i\}_{i\in I}$ be a family of simple combinatorial polytopes. A Hirzebruch genus $\phi$ is called \emph{combinatorially $\mc P$-rigid}, if for all pairs of toric varieties with moment polytopes combinatorially equivalent to the same polytope from~$\mc P$ one has $\phi(X_1)=\phi(X_2)$.
\end{definition}

With the definition of combinatorial $\mc P$-rigidity it is interesting to study the following problems:

{\medskip\noindent\bf Problem.} For a given family of combinatorial polytopes $\mc P$ (e.g., Stasheff polytopes, permutohedra, nestohedra, etc.) describe combinatorially $\mc P$-rigid Hirzebruch genera.

{\medskip\noindent\bf Problem.} For a given Hirzebruch genus $\phi$ find a maximal family of polytopes $\mc P$ such that the genus $\phi$ is combinatorially $\mc P$-rigid.

\section{Acknowledgements}
We are grateful to V.M.\,Buchstaber and T.E.\,Panov for suggesting the problems studied in this paper and for numerous fruitful discussions. We also would like to thank P.\,Landweber for the most helpful remarks and suggestions.


\end{document}